\documentclass{amsart}

\usepackage{latexsym}
\usepackage{amsxtra}
\usepackage{amsfonts}
\usepackage{xspace}
\usepackage{amssymb}
\usepackage{euscript}

\input xy
\xyoption{all} \CompileMatrices

\begin{document}

\def\II{{\mathfrak{I}}}
\def\PO{{\operatorname{POG}}}
\def\Cl{{\operatorname{Cl}}}
\def\Max{{\operatorname{-Max}}}
\def\XX{{\bf{X}}}
\def\YY{{\bf{Y}}}
\def\BBB{{\mathcal B}}
\def\inv{{\operatorname{inv}}}
\def\emph{\it}
\def\Int{{\operatorname{Int}}}
\def\Spec{\operatorname{Spec}}
\def\Bin{{\operatorname{B}}}
\def\n{\operatorname{b}}
\def\N{{\operatorname{GB}}}
\def\BC{{\operatorname{BC}}}
\def\dlog{\frac{d \log}{dT}}
\def\Sym{\operatorname{Sym}}
\def\Nr{\operatorname{Nr}}
\def\lbrack{{\{}}
\def\rbrack{{\}}}
\def\burnside{\operatorname{B}}
\def\Sym{\operatorname{Sym}}
\def\Hom{\operatorname{Hom}}
\def\Inj{\operatorname{Inj}}
\def\Aut{{\operatorname{Aut}}}
\def\Mor{{\operatorname{Mor}}}
\def\Map{{\operatorname{Map}}}
\def\CMap{{\operatorname{CMap}}}
\def\GMaps{G{\operatorname{-Maps}}}
\def\Fix{{\operatorname{Fix}}}
\def\res{{\operatorname{res}}}
\def\ind{{\operatorname{ind}}}
\def\inc{{\operatorname{inc}}}
\def\coind{{\operatorname{cnd}}}
\def\Equiv{{\mathcal{E}}}
\def\W{\operatorname{W}}
\def\F{\operatorname{F}}
\def\witt{\operatorname{gh}}
\def\ngh{\operatorname{ngh}}
\def\Fm{{\operatorname{Fm}}}
\def\bij{{\iota}}
\def\mk{{\operatorname{mk}}}
\def\km{{\operatorname{mk}}}
\def\VV{{\bf{V}}}
\def\ff{{\bf{f}}}
\def\ZZ{{\mathbb Z}}
\def\Zhat{{\widehat{\mathbb Z}}}
\def\CC{{\mathbb C}}
\def\PP{{\mathbb P}}
\def\JJ{{\mathbb J}}
\def\NN{{\mathbb N}}
\def\RR{{\mathbb R}}
\def\QQ{{\mathbb Q}}
\def\FF{{\mathbb F}}
\def\mm{{\mathfrak m}}
\def\nn{{\mathfrak n}}
\def\jj{{\mathfrak j}}
\def\aaa{{{{\mathfrak a}}}}
\def\bbb{{{{\mathfrak b}}}}
\def\ppp{{{{\mathfrak p}}}}
\def\qqq{{{{\mathfrak q}}}}
\def\PPP{{{{\mathfrak P}}}}
\def\QQQ{{{{\mathfrak Q}}}}
\def\MM{{\mathfrak M}}
\def\BB{{\mathfrak B}}
\def\jj{{\mathfrak J}}
\def\LL{{\mathfrak L}}
\def\qq{{\mathfrak Q}}
\def\rr{{\mathfrak R}}
\def\DD{{\mathfrak D}}
\def\cc{{\mathfrak S}}
\def\TT{{\mathcal{T}}}
\def\SS{{\mathcal S}}
\def\UU{{\mathcal U}}
\def\AA{{\mathcal A}}
\def\BB{{\mathcal B}}
\def\Primes{{\mathcal P}}
\def\genS{{\langle S \rangle}}
\def\genT{{\langle T \rangle}}
\def\bT{\mathsf{T}}
\def\bD{\mathsf{D}}
\def\bC{\mathsf{C}}
\def\VV{{\bf V}}
\def\ff{{\bf f}}
\def\uu{{\bf u}}
\def\aa{{\bf{a}}}
\def\bb{{\bf{b}}}
\def\gg{{\bf{g}}}
\def\zero{{\bf 0}}
\def\rad{\operatorname{rad}}
\def\End{\operatorname{End}}
\def\id{\operatorname{id}}
\def\mod{\operatorname{mod}}
\def\im{\operatorname{im}\,}
\def\ker{\operatorname{ker}}
\def\coker{\operatorname{coker}}
\def\ord{\operatorname{ord}}

\newtheorem{theorem}{Theorem}[section]
\newtheorem{proposition}[theorem]{Proposition}
\newtheorem{corollary}[theorem]{Corollary}
\newtheorem{conjecture}[theorem]{Conjecture}
\newtheorem{lemma}[theorem]{Lemma}
\theoremstyle{definition}
\newtheorem{definition}[theorem]{Definition}
\newtheorem{problem}[theorem]{Problem}
\newtheorem{remark}[theorem]{Remark}
\newtheorem{example}[theorem]{Example}

 \newenvironment{map}[1]
   {$$#1:\begin{array}{rcl}}
   {\end{array}$$
   \\[-0.5\baselineskip]
 }

 \newenvironment{map*}
   {\[\begin{array}{rcl}}
   {\end{array}\]
   \\[-0.5\baselineskip]
 }

 \newenvironment{nmap*}
   {\begin{eqnarray}\begin{array}{rcl}}
   {\end{array}\end{eqnarray}
   \\[-0.5\baselineskip]
 }

 \newenvironment{nmap}[1]
   {\begin{eqnarray}#1:\begin{array}{rcl}}
   {\end{array}\end{eqnarray}
   \\[-0.5\baselineskip]
 }

\newcommand{\eq}{eq.\@\xspace}
\newcommand{\eqs}{eqs.\@\xspace}
\newcommand{\diagram}{diag.\@\xspace}


\title{Presentations and module bases of integer-valued polynomial rings}

\author{Jesse Elliott}
\date{\today} \address{Department of Mathematics\\ California
State University, Channel Islands\\ Camarillo, California 93012}
\email{jesse.elliott@csuci.edu}

\maketitle

\begin{abstract}

Let $D$ be an integral domain  with quotient field $K$.   For any set $\XX$, the ring $\Int(D^\XX)$ of {\it integer-valued polynomials on $D^\XX$} is the set of all polynomials $f \in K[\XX]$ such that $f(D^\XX) \subseteq D$.  Using the $t$-closure operation on fractional ideals, we find for any set $\XX$ a $D$-algebra presentation of $\Int(D^\XX)$ by generators and relations for a large class of domains $D$,  including any unique factorization domain $D$, and more generally any Krull domain $D$ such that $\Int(D)$ has a {\it regular basis}, that is, a $D$-module basis consisting of exactly one polynomial of each degree.  As a corollary we find for all such domains $D$ an intrinsic characterization of the $D$-algebras that are isomorphic to a  quotient of $\Int(D^\XX)$ for some set $\XX$.  We also generalize the well-known result that a Krull domain $D$ has a regular basis if and only if the P\'olya-Ostrowski group of $D$ (that is, the subgroup of the class group of $D$ generated by the images of the factorial ideals of $D$) is trivial, if and only if the product of the height one prime ideals of finite norm $q$ is principal for every $q$.
\end{abstract}





\section{Introduction}\label{introduction}

Let $D$ be an integral domain with quotient field $K$.  The ring of {\it integer-valued polynomials on $D$} is the subring
$$\Int(D) = \{f \in K[X] : f(D) \subseteq D\}$$ of the polynomial ring $K[X]$.  More generally, if $\XX$ is a set, then
the ring of {\it integer-valued polynomials on $D^\XX$} is the subring $$\Int(D^\XX) = \{f \in K[\XX] : f(D^\XX) \subseteq D\}$$ of $K[\XX]$ \cite{cah}.  The study of integer-valued polynomial rings---on number rings---began with P\'olya and Ostrowski circa 1919 \cite[p.\ xiv]{cah}.  They showed that, for any number ring $D$, the $D$-module $\Int(D)$ has a {\it regular basis}, that is, a $D$-module basis consisting of exactly one polynomial of each degree, if and only if the product $\Pi_q$ of the prime ideals of $D$ of norm $q$ is a principal ideal for every $q$.  In fact this equivalence holds for any Dedekind domain $D$.  More generally, if $D$ is a Krull domain, then $\Int(D)$ has a regular basis if and only if the product $\Pi_q$ of the height one prime ideals of norm $q$ is principal for every $q$ \cite[Corollary 2.5]{cha}.  In particular, if $D$ is a unique factorization domain, then $\Int(D)$ has a regular basis.  Moreover, for any Krull domain $D$, there is a subgroup $\PO(D)$ of the class group $\Cl(D)$ of $D$, generated by the images of the so-called {\it factorial ideals} $n!_D$ of $D$, that in some sense measures the extent to which $\Int(D)$ fails to have a regular basis; specifically, $\Int(D)$ has a regular basis if and only if the group $\PO(D)$ is trivial \cite[Corollary 2.5]{cha}.   One of our main results, Theorem \ref{equivthm} (in Section \ref{sec:4}), generalizes these results on Krull domains to a much larger class of integral domains, including the domains of Krull type (equivalently the Pr\"ufer $v$-multiplication domains (PVMDs) of finite $t$-character), hence the TV PVMDs. (The latter classes of domains are defined in Sections \ref{sec:3} and \ref{sec:4}.)

Since P\'olya's and Ostrowski's seminal work, much attention has been given to finding $D$-module bases of integer-valued polynomial rings.   For any Dedekind domain  $D$ for which $\Int(D)$ has a regular basis, \cite[Proposition II.3.14]{cah} provides an algorithm to construct any finite number of elements of such a basis.  (Theorem \ref{regbasisalg} generalizes this algorithm.)  Moreover, \cite[Corollary 3.11]{cha} provides a characterization of all cyclic number fields $K$ such that $\Int(\mathcal{O}_K)$ has a regular basis, and \cite[Corollary II.4.5]{cah} and \cite[Propositions 3.4, 3.6, and 3.19]{arm} explicitly construct all such $K$  of degree $2$, $3$, $4$, and $6$ over $\QQ$. The number field $K = \QQ(\sqrt{-5})$ is an example where $\PO(\mathcal{O}_K)$ has order $2$; so is $K = \QQ(\sqrt{-29})$, where one also has $\PO(\mathcal{O}_K) \subsetneq \Cl(\mathcal{O}_K)$ \cite[Exercise II.31]{cah}.  Unfortunately we do not know a characterization of the number fields $K$ such that $\PO(\mathcal{O}_K)$ has order $2$.

Although  $\Int(D)$ does not have a regular basis for many (and probably most) number rings $D$, the $D$-module $\Int(D)$ is free for any Dedekind domain $D$ \cite[Remark II.3.7(iii)]{cah}.  Surprisingly, however, there are no confirmed examples in the literature  (of which we are aware) of an integral domain $D$ such that $\Int(D)$ is not free as a $D$-module.  In an earlier paper \cite{ell2}, we showed that $\Int(D)$ is locally free if $D$ is a Krull domain, or more generally if $D$ is a TV PVMD \cite[Theorem 1.2]{ell2}.  We also conjectured that $\Int(D)$ is not flat as a $D$-module for $D = \FF_2[[T^2, T^3]]$ and for $D = \FF_2 +T\FF_4[[T]]$.  This conjecture is still open.

In this paper we also consider a related problem, that of finding a $D$-algebra presentation of $\Int(D)$ by generators and relations.  This problem is motivated by results in the existing literature on integer-valued polynomial rings as follows.  First, a presentation for $\Int(\ZZ^\XX)$ for any set $\XX$ is given in \cite{jess2}, where the presentation is used to provide several equivalent conditions for a ring to be {\it binomial} in the sense of \cite{wil}.  In \cite{des} it is shown that, for any finite extension $K$ of the field $\QQ_p$ of $p$-adic rational numbers, one can construct from any Lubin-Tate formal group law $F \in  \mathcal{O}_K[[X,Y]]$ a minimal set $\{f_{i}: i \geq 0\}$ of generators of $\Int(\mathcal{O}_K)$ as an $\mathcal{O}_K$-algebra.  (For example, if $K = \QQ_p$ and $F = X+Y+XY$, then $f_i = {X \choose p^i}$ for all $n$.)  However, we do not know a complete set of relations for the generators $f_i$.  A more well-known result, \cite[Proposition II.3.14]{cah}, implies that, for any Dedekind domain $D$ such that $\Int(D)$ has a regular basis, $\Int(D)$ is generated as a $D$-algebra by the {\it $q^n$th $q$-Fermat polynomial} $F_q^{\circ n} = F_q \circ F_q \circ \cdots \circ F_q$ for all positive integers $n$ and all prime powers $q$, where $F_q = \frac{X^q-X}{\pi_q}$ and $\pi_q$ is any generator of $\Pi_q$.   The question of the relations among these generators has not been raised.  To this end we show in Theorem \ref{presentation3} that the obvious relations $(F_q^{\circ n})^q - F_q^{\circ n}  = \pi_q F_q^{\circ (n+1)}$ are a complete set of relations for the generators $F_q^{\circ n}$ of $\Int(D)$.  More generally, Theorem \ref{presentation3}, for a large class $\mathcal{C}$ of domains $D$, including the Krull domains $D$ for which $\Int(D)$ has a regular basis, provides for any set $\XX$ a complete set of generators and relations for the $D$-algebra $\Int(D^\XX)$, using an apporpriate generalization of the $q$-Fermat polynomials.

A nontrivial application of this algebra presentation of integer-valued polynomial rings is as follows.  It is known that a ring $A$ is isomorphic to a quotient of $\Int(\ZZ^\XX)$ for some set $\XX$ if and only if the endomorphism $a \longmapsto a^p$ of $A/pA$ is the identity for every prime number $p$ \cite[Theorem 4.1]{jess2}.  Theorem \ref{polybthm} generalizes this by showing that, for any integral domain $D$ and any (commutative) $D$-algebra $A$, if $D$ is in the class $\mathcal{C}$ mentioned above or if $A$ is $D$-torsion-free, then the following conditions are equivalent.
\begin{enumerate}
\item $A$ is isomorphic to a quotient of $\Int(D^\XX)$ for some set $\XX$.
\item For every $a \in A$ there is a $D$-algebra homomorphism $\Int(D) \longrightarrow A$ sending $X \in \Int(D)$ to $a$.
\item The endomorphism $a \longmapsto a^{N(\ppp)}$ of $A/\ppp A$ is the identity for every $t$-maximal prime ideal $\ppp$ of $D$ of finite norm $N(\ppp) = |D/\ppp|$.  (The $t$-maximal ideals of an integral domain are defined in Section \ref{sec:3}.)
\end{enumerate} 
Our proof of the equivalence of conditions (2) and (3) above for domains in the class $\mathcal{C}$ uses the presentation of $\Int(D)$ mentioned above in an essential way;  for this reason we suspect that the equivalence  does not hold for all Dedekind domains $D$ and all $D$-algebras $A$, although we do not know a counterexample.

One of the main tools we use in this paper is that of a star operation, or  $'$-operation, introduced by Krull in \cite[Section 6.43]{kru}, on fractional ideals.  Specifically, we use the well-known star operations of divisorial closure, $t$-closure, and $w$-closure.  These are immensely useful for generalizing known results on Dedekind domains and Noetherian domains to larger classes of domains. All of the definitions and facts we need are summarized in Section \ref{sec:3}; proofs can be found in \cite{gil}, which is a classic reference on multiplicative ideal theory.

The main results in this paper---those results labeled ``Theorem''---are Theorems \ref{factprop}, \ref{trivcor}, \ref{regbasisalg}, \ref{equivthm}, \ref{presentation3}, \ref{polyathm}, \ref{polybthm}, and \ref{numthm}.  In Section \ref{sec:2} we provide a $D$-algebra presentation for $\Int(D)$ when $D$ is a finite dimensional local domain with principal maximal ideal.   In Sections \ref{sec:3} and \ref{sec:4} we generalize the results in \cite[Sections II.1 and II.3]{cah}, which focus on Dedekind domains, to much larger classes of domains.  There we define and prove new results on the characteristic ideals and factorial ideals of a domain, as well as its P\'olya-Ostrowski group, which we define when the factorial ideals are $t$-invertible.  In Section \ref{sec:5} we find a $D$-algebra presentation of $D$ when $D$ is in a large class $\mathcal{C}$ of integral domains, including all Krull domains such that $\Int(D)$ has a regular basis.  Finally, in Sections \ref{sec:6} through \ref{sec:8} we apply our previous results to the study of $D$-algebras that are isomorphic to $\Int(D^\XX)$ for some set $\XX$, and also to $D$-algebras $A$ such that for every $a \in A$ there is a $D$-algebra homomorphism $\Int(D) \longrightarrow A$ sending $X \in \Int(D)$ to $a$.  

All rings and algebras in this paper are commutative with identity.  For any ring $R$, and for any $f \in R[X]$ and any nonnegative integer $n$, we let $f^{\circ n}$ denote the $n$-fold composition of $f$ with itself, where $f^{\circ 0} = X$.   

\section{The local case}\label{sec:2}

In this section we find a $D$-algebra presentation for $\Int(D)$ when $D$ is a finite dimensional local domain with principal maximal ideal. 

\begin{lemma}\label{nonzerodiv}
If $B \supseteq A$ is an integral extension of rings, and if $a \in A$ is
a nonzerodivisor, then $a$ is a nonzerodivisor in $B$.
\end{lemma}

\begin{proof}
This is clear.
\end{proof}

\begin{proposition}\label{presentation}
Let $D$ be a local domain with principal maximal ideal $\pi D$ and finite residue field of order $q$.  Let $F_q = \frac{X^q - X}{\pi} \in \Int(D)$.  The unique $D$-algebra homomorphism
$$\varphi: {\begin{array}{rrr} D[X_0, X_1, X_2, \ldots] & \longrightarrow & \Int(D)  \\
 X_{k} & \longmapsto & F_{q}^{\circ k}
\end{array}}$$
is surjective, and if $D$ has finite Krull dimension then $\ker \varphi$ is equal to the ideal $I$ generated by $X_{k}^q - X_{k} - \pi X_{k+1}$ for
all $k \in \ZZ_{\geq 0}$.
\end{proposition}

\begin{proof}
The homomorphism $\varphi$ is surjective by  \cite[Remark II.2.14]{cah} and the proof of \cite[Proposition II.3.14]{cah}.    Suppose that $D$ has finite Krull dimension.
For any positive integer $n$, let
$$A_n = D[X_{0}, X_{1}, X_{2}, \ldots, X_{n}]$$ and $$B_n =
D[X, F_q, F_q^{\circ 2}, \ldots, F_q^{\circ n}].$$ 
One has $\varphi(A_n) = B_n$, so $\varphi$ restricts to a surjective homomorphism $\varphi_n: A_n \longrightarrow B_n$.  Moreover, one has $\ker \varphi = \bigcup_n \ker \varphi_n$. 
It therefore suffices to show that $\ker \varphi_n = J_n$, where $J_n$ is the ideal in
$A$ generated by $X_{k}^q - X_{k} - \pi X_{k+1}$ for $0 \leq k \leq n-1$.  Clearly $J_n$ is
contained in $\ker \varphi_n$, so we have a surjective ring
homomoprhism
$$\psi_n: A_n^\prime \longrightarrow B_n,$$ where $A_n^\prime = A_n/J_n$.  We must show that
$\ker \psi_n = 0$.  Now, $A_n^\prime$ is integral over $D[X_n]$, and likewise $B_n$ is integral over $D[F_q^{\circ n}] \cong D[X]$.  Therefore by \cite[Exercise 11.6]{ati} both rings have Krull dimension $\dim D[X] \leq 1+ 2 \dim D < \infty$.  Thus, since $B_n$ is an integral domain, the kernel of $\psi_n$ must be a mimimal prime
ideal in $A_n^\prime$. But by Lemma \ref{nonzerodiv}, $\pi$ is a nonzerodivisor in $A_n^\prime$,
so the map $$A_n^\prime \longrightarrow A_n^\prime[\pi^{-1}] = D[\pi^{-1}][X_0]$$ is an inclusion of
rings.  Therefore $A_n^\prime$ is a domain, so the kernel of $\psi_n$, being a
minimal prime in $A^\prime_n$, is zero.
\end{proof}

 We do not know if the hypothesis in Proposition \ref{presentation} of the finite dimensionality of $D$ is necessary.

\begin{remark}
If $D$ is the ring of integers $\mathcal{O}_K$ for some finite extension of the field $\QQ_p$ of $p$-adic rational numbers, then for any Lubin-Tate formal group law $F \in D[[X,Y]]$ over $D$ and for any $a \in D$, there is a unique formal power series $$[a]_F(T) = \sum_{n = 1} c_n(a) T^n \in D[[T]]$$
such that $ [a]_F(F(X,Y)) = F([a]_F(X),[a]_F(Y))$ in $D[[X,Y]]$ and $c_1(a) = a$; moreover, for each $n$ one has $c_n(a) = f_n(a)$ for all $a \in D$ for a unique $f_n \in \Int(D)$, and $\deg f_n \leq n$ \cite{des}.  By \cite[Theorem 3.1]{des} one has $$\Int(D) = D[f_1, f_2, f_3, \ldots],$$ and in fact $\{f_{q^i}: i \geq 0\}$ is a minimal set of generators of $\Int(D)$ as a $D$-algebra, where $q$ is the cardinality of the residue field of $D$.   For example, if $K = \QQ_p$ and $F = X+Y+XY$, then $f_n = {X \choose n}$ for all $n$, and in that case a complete set of relations among the $f_n$ is known.  However, for general $K$ and $F$ we do not know a complete set of relations for the $f_n$.  Such a complete set of relations would provide an alternative $D$-algebra presentation for the ring $\Int(D)$ in this case.
\end{remark}

\section{Regular bases and characteristic and factorial ideals}\label{sec:3}

Let $D$ be an integral domain with quotient field $K$.  A {\it regular basis} of $\Int(D)$ is a $D$-module basis of $\Int(D)$ consisting of exactly one polynomial of each degree.  By \cite[Corollary 2.5]{cha}, if $D$ is a Krull domain, then $\Int(D)$ has a regular basis if and only if the product $\Pi_q$ of all height one prime ideals $\ppp$ of $D$ with $|D/\ppp| = q$ is a principal ideal for every prime power $q$.  In particular, $\Int(D)$ has a regular basis if $D$ is a unique factorization domain, since in that case every height one prime ideal of $D$ is principal.  In this section and the next we find a more general characterization for a much larger class of domains, including, for example, all TV PVMDs, using the $t$-closure operation on fractional ideals, described below.

A {\it fractional ideal} of $D$ is a $D$-submodule $I$ of $K$ such that $I^{-1} = (D :_K I)$ is nonzero, or equivalently such that $aI \subseteq D$ for some nonzero $a \in D$.  A {\it star operation} on $D$ is a closure operation $*$ on the partially ordered set $\mathcal{F}(D)$ of nonzero fractional ideals of $D$ such that $D^* = D$ and $I^* J^* \subseteq (IJ)^*$ for all $I,J \in \mathcal{F}(D)$.  Equivalently, a star operation on $D$ is a self-map $*$ of $\mathcal{F}(D)$ satisfying the following conditions for all $I,J \in \mathcal{F}(D)$ and all nonzero $a \in D$.
\begin{enumerate}
\item $I \subseteq I^* = (I^*)^*$, and $I \subseteq J \Longleftrightarrow I^* \subseteq J^*$.
\item $D^* = D$.
\item  $(aI)^* = aI^*$.
\end{enumerate}
For any star operation $*$ one also has the following.
\begin{enumerate}
\item[(4)] $(I^* J^*)^* = (IJ)^*$.
\item[(5)] $(I^*+J^*)^* = (I+J)^*$.
\item[(6)] $(I^* \cap J^*)^* = I^* \cap J^*$.
\end{enumerate}

We will make use of the following important star operations.  First, the {\it $d$-closure} star operation $d$ is the identity operation.  The {\it divisorial closure} star operation $v$, also known as the {\it $v$-operation}, acts by $v: I \longrightarrow I_v = (I^{-1})^{-1}$.  The {\it $t$-closure} star operation $t$ acts by $$t: I \longrightarrow I_t = \bigcup \{J_v: J \subseteq I \mbox{ and } J \mbox{ is finitely generated}\} .$$  Finally, the {\it $w$-closure} star operation $w$ acts by
$$w: I \longrightarrow I_w = \bigcup \{(I :_K J): J \subseteq D \textup{ and } J_t = D\}.$$  One has $d \leq w \leq t \leq v$, where one writes $*_1 \leq *_2$ if $I^{*_1} \subseteq I^{*_2}$ for all $I \in \mathcal{F}(D)$.  Under this partial ordering of star operations, $d$ is the smallest star operation on $D$; $v$ is the largest star operation on $D$; $t$ is the largest finite type star operation on $D$, where $*$ is of {\it finite type} if $I^* = \bigcup \{J^*: J \subseteq I \mbox{ and } J \mbox{ is finitely generated}\}$ for all $I$; and $w$ is the largest stable finite type star operation on $D$, where $*$ is {\it stable} if $(I \cap J)^* = I^* \cap J^*$ for all $I,J$. 

If $v = d$ on $\mathcal{F}(D)$, or equivalently, if $d$ is the only star operation on $D$, then $D$ is said to be {\it divisorial}.   For example, a Dedekind domain is equivalently a divisorial Krull domain.  If $t = v$ on $\mathcal{F}(D)$, or equivalently, if $v$ is of finite type, then $D$ is said to be a {\it TV domain}.  Any Noetherian or Krull domain is a TV domain.

Let $*$ be a star operation on $D$.  A fractional ideal $I$ of $D$ is said to be a {\it $*$-ideal} if $I^* = I$, and an {\it integral $*$-ideal} if $I$ is a $*$-ideal contained in $D$.  For example, every invertible fractional ideal of $D$ is a $*$-ideal.  An ideal of $D$ that is maximal among the proper integral $*$-ideals of $D$ is said to be {\it $*$-maximal}.  Every $*$-maximal ideal of $D$ is prime.  Moreover, if $*$ is of finite type, then every proper integral $*$-ideal of $D$ is contained in some $*$-maximal ideal of $D$.  We let $*\Max(D)$ denote the set of all $*$-maximal ideals of $D$, which is nonempty if $*$ is of finite type.

The operation $(I,J) \longmapsto (IJ)^*$ on $\mathcal{F}(D)$ is called {\it $*$-multiplication}.  The set of all $*$-ideals of $D$ is a partially ordered monoid under $*$-multiplication.  Its group of units is the group of $*$-invertible $*$-ideals, where a fractional ideal $I$ is {\it $*$-invertible} if $(IJ)^* = D$ for some fractional ideal $J$, in which case $(II^{-1})^* = D$ and $I^{-1} = J^*$ is the inverse of $I$ under $*$-multiplication. The {\it $*$-class group} $\Cl^*(D)$ of $D$ is the group of $*$-invertible $*$-ideals of $D$ under $*$-multiplication modulo the subgroup of principal fractional ideals of $D$.  Since any invertible ideal of $D$ is a $*$-invertible $*$-ideal, one has $\operatorname{Pic}(D) \subseteq \Cl^*(D)$, where $\operatorname{Pic}(D) = \Cl_d(D)$ is the Picard group of $D$.  Here we will be particularly interested in the $t$-class group (or $w$-class group) $\Cl_t(D) = \Cl_w(D)$, which in general carries more information than the classical object $\operatorname{Pic}(D)$.  Note that a $*$-invertible $*$-ideal is a $v$-invertible $v$-ideal and $\Cl^*(D)$ is a subgroup of $\Cl_v(D)$; and if $*$ is of finite type, then a $*$-invertible $*$-ideal is a $t$-invertible $t$-ideal and $\Cl^*(D)$ is a subgroup of $\Cl_t(D)$.

 A domain $D$ is a Krull domain if and only if every nonzero fractional ideal of $D$ is $t$-invertible.  A domain $D$ is a unique factorization domain if and only if the $t$-closure of every nonzero fractional ideal of $D$ is principal, if and only if $D$ is a Krull domain with trivial $t$-class group.   For any Krull domain $D$, one has $v = t =  w$ on $\mathcal{F}(D)$; a fractional ideal of $D$ is a $t$-ideal if and only if it is divisorial; an ideal of $D$ is $t$-maximal if and only if it is  a prime ideal of height one; and the $t$-class group $\Cl_t(D)$ is equal to the usual class group $\Cl(D)$ of $D$. 

 For any fractional ideal $I$ of a domain $D$ one has $t\Max(D) = w\Max(D)$ and $I_w = \bigcap_{\ppp \in t\Max(D)} I D_\ppp$, and in particular $D = \bigcap_{\ppp \in t\Max(D)} D_\ppp$.  In fact one has $A = \bigcap_{\ppp \in t\Max(D)} A_\ppp$ for any flat extension $A$ of $D$. Any weak Bourbaki associated prime of the $D$-module $K/D$, and more generally any Nortcott attached prime of $K/D$, is a prime $t$-ideal; and a prime ideal $\ppp$ of $D$ is a Northcott attached prime of $K/D$ if and only if $\ppp D_\ppp$ is a $t$-maximal ideal of $D_\ppp$ \cite[Lemma 2.2]{ell2}.  Moreover, if $D$ is a TV domain (or a Noetherian or Krull domain), then every $t$-maximal ideal of $D$ is a weak Bourbaki associated prime of $K/D$.  

A nonzero fractional ideal $I$ of $D$ is $t$-invertible if and only if $I$ is $w$-invertible, if and only if $I_t = J_t$ for some finitely generated fractional ideal $J$ of $D$ and $I_t D_\ppp$ is principal for every $t$-maximal ideal $\ppp$ of $D$.  If $D$ is a {\it Mori domain}, that is, if $D$ satisfies the ascending chain condition on integral $v$-ideals, then a nonzero fractional ideal $I$ of $D$ is $t$-invertible if and only if $I_v D_\ppp$ is principal for every weak Bourbaki associated prime $\ppp$ of $D$.  Any Noetherian or Krull domain is Mori, and any Mori domain is a TV domain.  A domain $D$ is said to be a {\it PVMD} if every finitely generated ideal of $D$ is $t$-invertible, which holds if and only if $D$ is integrally closed and $t = w$ on $\mathcal{F}(D)$, if and only if $D_\ppp$ is a valuation domain for every $t$-maximal ideal $\ppp$ of $D$.  A Krull domain is equivalently a Mori PVMD.

The reader is referred to \cite{gil} for proofs of the facts listed above.  

\begin{definition}
Let $D$ be an integral domain with quotient field $K$ and $n$ be a nonnegative integer.  The {\it norm} $N(\aaa)$ of an ideal $\aaa$ of $D$ is defined to be the cardinality $|D/\aaa|$ of the quotient ring $D/\aaa$.  
Let $\Pi_n(D)$ denote the ideal $$\Pi_n(D) = \bigcap_{{\ppp \in t\Max(D)} \atop {N(\ppp) = n}} \ppp$$ of $D$; that is, let $\Pi_n(D)$ denote the intersection of all $t$-maximal ideals $\ppp$ of $D$ of norm $n$ (which is equal to $D$ if $n$ is not a power of a prime). 
Following \cite[Chapter II]{cah} we let $\Int_n(D)$ denote the $D$-submodule $$\Int_n(D)  = \{f \in \Int(D): \deg f \leq n\}$$ of $\Int(D)$, and we let $\II_n(D)$ denote the $D$-submodule
$$\II_n(D) = \left\{f_n: f = \sum_{i = 0}^n f_i X^i \in \Int_n(D)\right\}$$ of $K$  consisting of the coefficient of $X^n$ for each polynomial $f \in \Int(D)$ of degree at most $n$.   By \cite[Proposition I.3.1]{cah} $\II_n(D)$ is a fractional ideal of $D$; it is called the {\it $n$th characteristic (fractional) ideal of $D$}.  Following \cite[Definition 1.2]{cha2} we let $$n!_D = \II_n(D)^{-1},$$
which is an integral $v$-ideal (hence $t$-ideal) of $D$ called the {\it $n$th factorial ideal of $D$}.
\end{definition}

We will write $\Pi_n = \Pi_n(D)$ and $\II_n = \II_n(D)$ when the domain $D$ is understood.   By \cite[Proposition 7.3]{ell3}, the ideal $(D[X]:_D \Int_n(D))$ of $D$ is a (nonzero) $v$-ideal, and therefore $\Int_n(D)$ lies between two free $D$-modules of rank $n+1$.   Note that $n!_D$ contains $(D[X]:_D \Int_n(D))$; moreover, equality holds, and also $\II_n = n!_D^{-1}$, if $D$ is a Krull domain or $\II_n$ is invertible for all $n$, by \cite[Remark 2.6]{cha} and \cite[Proposition II.1.7]{cah}.  We will show, more generally, that $n!_D = (D[X]:_D \Int_n(D))$ and $\II_n = n!_D^{-1}$ if $\II_n$ is $t$-invertible for all $n$, which holds, for example, if $D$ is a TV PVMD.  (See Theorem \ref{trivcor} and Corollary \ref{pvmdcor}.)  The following well-known result explains the significance of the characteristic ideals.

\begin{proposition}[{\cite[Proposition II.1.4]{cah}}]\label{charideal}
Let $D$ be an integral domain.  Then $\Int(D)$ has a regular basis if and only if the $n$th characteristic ideal $\II_n(D)$ of $D$ is principal for every $n$.  In fact, a set $\{f_0, f_1, f_2, \ldots\}$ of elements of $\Int(D)$ with $\deg f_n  = n$ for all $n$ is a regular basis of $\Int(D)$ if and only if $\II_n(D) = a_n D$ for all $n$, where $a_n$ is the leading coefficient of $f_n$.
\end{proposition}


Next we examine properties of the characteristic ideals $\II_n(D)$ and some consequences of the assumption that they are $t$-invertible.

\begin{lemma}\label{tinv}
If a nonzero fractional ideal $I$ of an integral domain $D$ is $t$-invertible, then
$D_\ppp = II^{-1} D_\ppp$ and $(ID_\ppp)^{-1} = I^{-1} D_\ppp$ for every $t$-maximal ideal $\ppp$ of $D$.
\end{lemma}

\begin{proof}
It is well-known that $I$ is $t$-invertible if and only if $I$ is $w$-invertible.
Suppose that $I$ is $w$-invertible, so $$D = (II^{-1})_w = \bigcap_{\ppp \in t\Max(D)} II^{-1}D_\ppp.$$ It follows that
$$D_\ppp =  II^{-1}D_\ppp = (ID_\ppp)( I^{-1}D_\ppp),$$
and therefore $(ID_\ppp)^{-1} = I^{-1} D_\ppp$, for all $\ppp \in t\Max(D)$.  
\end{proof}


For any multiplicative subset $S$ of a domain $D$, one has $S^{-1}\Int(D) \subseteq \Int(S^{-1}D)$, by \cite[Proposition I.2.2]{cah}; however, by \cite[Example VI.4.15]{cah} the reverse inclusion need not hold, even if $D$ is locally a discrete valuation ring (DVR).  Following \cite{ell2}, we say that a domain $D$ is {\it polynomially L-regular} if $S^{-1}\Int(D) = \Int(S^{-1}D)$ for every multiplicative subset $S$ of $D$.  For example,  by \cite[Propositon 2.4]{ell2} any TV domain (hence any Mori domain, hence any Noetherian domain) is polynomially L-regular.  Equivalent conditions for a domain to be polynomially L-regular are given in \cite[Proposition 2.3]{ell2}.

\begin{theorem}\label{factprop}
Let $D$ be an integral domain.
\begin{enumerate}
\item $D = \II_0(D) = \II_1(D) \subseteq \II_2(D) \subseteq \II_3(D) \subseteq \cdots$ and $\II_k(D) \II_l(D) \subseteq \II_{k+l}(D)$ for all $k,l$.
\item $D = 0!_D = 1!_D \supseteq 2!_D \supseteq 3!_D \supseteq \cdots$ and $(k+l)!_D \subseteq (k!_D^{-1} l!_D^{-1})^{-1}$ for all $k,l$.
\item For any multiplicative subset $S$ of $D$, one has $S^{-1}\Int(D) = \Int(S^{-1}D)$ if and only if $S^{-1}\II_n(D) = \II_n(S^{-1}D)$ for every nonnegative integer $n$.
\item If $\ppp$ is a prime ideal of $D$ that is not $t$-maximal of finite norm, then $\Int(D)_\ppp = \Int(D_\ppp) = D_\ppp[X]$ and $\II_n(D)_\ppp = \II_n(D_\ppp) = D_\ppp$ for every nonnegative integer $n$.
\item $\II_n(D)$ is a $w$-ideal and $n!_D$ is a $v$-ideal for every nonnegative integer $n$.
\item  If $\II_n(D)$ is $t$-invertible, where $n$ is a nonnegative integer, then $\II_n(D)$ is a $v$-ideal and $\II_n(D) = n!_D^{-1}$.
\item Suppose that $\II_k(D)$ is $t$-invertible for all $k \leq n$.
\begin{enumerate}
\item For any $f \in \Int_n(D)$, all of the coefficients of $f$ lie in $\II_n(D)$.
\item  $n!_D = (D[X]:_D \Int_n(D))$.
\item $(k+l)!_D \subseteq (k!_D l!_D)_t$ for all $k,l \leq n$.
\end{enumerate}
\item  If $D$ is polynomially L-regular, then one has the following.
\begin{enumerate}
\item $S^{-1}\II_n(D) = \II_n(S^{-1}D)$ for every $n$ and every multiplicative subset $S$ of $D$.
\item  If $\II_n(D)$ is $t$-invertible, where $n$ is a nonnegative integer, then $n!_D D_\ppp = n!_{D_\ppp}$ for every $t$-maximal ideal $\ppp$ of $D$, and $\II_n(D_\ppp)$ is principal for every prime ideal $\ppp$ of $D$.
\item If $\II_n(D)$ is $t$-invertible for all $n$, then  $\Int(D_\ppp)$ has a regular basis for every prime ideal $\ppp$ of $D$ and the $D$-module $\Int(D)$ is locally free.
\end{enumerate}
\end{enumerate}
\end{theorem}

\begin{proof}
Statements (1) and (2) are clear.  Suppose that $S^{-1}\II_n(D) = \II_n(S^{-1}D)$ for every nonnegative integer $n$.   Let $f \in \Int(S^{-1}D)$, and let $n = \deg f$.  Since $f_n \in \II_n(S^{-1}D) = S^{-1}\II_n(D)$, one has $f_n = g_n/u$, where $g = \sum_{i = 0}^n g_iX^i \in \Int_n(D)$ and $u \in S$.  Then $f - g/u \in \Int_{n-1}(S^{-1}D)$, so by induction on $n$ we may assume $f-g/u \in S^{-1}\Int_{n-1}(D)$.  Therefore $f \in S^{-1}\Int_n(D)$.  Thus we have $S^{-1}\Int(D) = \Int(S^{-1}D)$, assuming that $S^{-1}\II_n(D) = \II_n(S^{-1}D)$ for every nonnegative integer $n$.  Since the converse is clear, this proves (3).

Next, let $\ppp$ be a prime ideal of $D$ that is not $t$-maximal of finite norm. 
Then $\Int(D)_\ppp = \Int(D_\ppp) = D_\ppp[X]$ by \cite[Lemma 2.2]{ell2}, whence $\II_n(D)_\ppp = \II_n(D_\ppp) = D_\ppp$ for all $n$ by statement (3).   This proves (4).

Now let $\mathcal{S} = \operatorname{Spec}(D) \backslash t\Max(D)$.  One has
\begin{eqnarray*}
\II_n(D)_w & = & \bigcap_{\ppp \in t\Max(D)} \II_n(D)_\ppp \\
 & \subseteq & \bigcap_{\qqq \in \mathcal{S}} \bigcap_{\ppp \in t\Max(D)}  (\II_n(D)_\qqq)_\ppp  \\
 & = &  \bigcap_{\qqq \in \mathcal{S}} \bigcap_{\ppp \in t\Max(D)} (D_\qqq)_\ppp  \\
 & = &  \bigcap_{\qqq \in \mathcal{S}} D_\qqq. 
\end{eqnarray*}
Therefore $$ \II_n(D)_w = \left( \bigcap_{\ppp \in t\Max(D)} \II_n(D)_\ppp \right) \cap \left(\bigcap_{\qqq \in \mathcal{S}} D_\qqq \right)  = \bigcap_{\ppp \in \operatorname{Spec}(D)} \II_n(D)_\ppp = \II_n(D).$$
Moreover, one has $(n!_D)_v = ((\II_n(D)^{-1})_v = \II_n(D)^{-1} = n!_D$.
This proves (5).

Next, suppose that $\II_n(D)$ is $t$-invertible, where $n$ is a nonnegative integer.  Then $\II_n(D)$ is a $w$-invertible $w$-ideal, hence a $v$-ideal.  Therefore $\II_n(D) = \II_n(D)_v = n!_D^{-1}$.    This proves (6).

We prove (7)(a) by induction.  Suppose that $\II_k = \II_k(D)$ is $t$-invertible, hence $w$-invertible, for all $k \leq n$.   If $f \in \Int(D)$ is constant, then $f \in D = \II_0$.  Let $f = \sum_{i = 0}^n f_i X^i \in \Int_n(D)$.  Consider the fractional ideal $I = \II_{n-1}\II_n^{-1}$ of $D$.  Since $\II_{n-1} \subseteq \II_n$, one has $I \subseteq D$.  Let $a \in I$.  One has $a f_n \in I \II_n \subseteq \II_{n-1}$, so there exists $g = \sum_{i = 0}^{n-1} g_i X^i \in \Int_{n-1}(D)$ with $a f_n  = g_{n-1} \in \II_{n-1}$.  The polynomial $h = a f - Xg$  lies in $\Int(D)$ and has degree at most $n-1$.  By the induction hypothesis, one has $g_{i-1} \in \II_{n-1}$ and $h_i =  a f_i -g_{i-1} \in \II_{n-1}$, and therefore $a f_i \in \II_{n-1}$, for $0 \leq i \leq n-1$.  Therefore $I f_i \subseteq \II_{n-1}$ for all $i \leq n$.  Since $I$ is $w$-invertible (being the product of two $w$-invertible fractional ideals), we have
$$f_i D = (f_i II^{-1})_w \subseteq (\II_{n-1}(\II_{n-1}\II_n^{-1})^{-1})_w = (\II_n)_w = \II_n,$$
and therefore $f_i \in \II_n$, for all $i \leq n$.    Thus all of the coefficients of $f$ lie in $\II_n$.  This proves (7)(a).  It follows that, for any $r \in n!_D$, one has $r \II_n \subseteq D$, whence $r \Int_n(D) \subseteq D[X]$ and so $r \in (D[X]:_D \Int_n(D))$.  Therefore $n!_D = (D[X]:_D \Int_n(D))$.  This proves (7)(b).  To prove (7)(c), note that by (2) one has $(k+l)!_D \subseteq (k!_D^{-1}l!_D^{-1})^{-1} = (k!_D l!_D)_t.$

Suppose now that $D$ is polynomially L-regular.  Statement (8)(a) follows from statement (3).  Suppose that $\II_n(D)$ is $t$-invertible, and let $\ppp$ be a prime ideal of $D$.  If $\ppp$ is $t$-maximal, then $(\II_n(D))_t D_\ppp = \II_n(D_\ppp)$ is principal, and by Lemma \ref{tinv} one also has $$n!_{D_\ppp} = \II_n(D_\ppp)^{-1} = (\II_n(D)D_\ppp)^{-1} = \II_n(D)^{-1}D_\ppp =  n!_D D_\ppp.$$ 
If $\ppp$ is not $t$-maximal, then $\II_n(D_\ppp) = D_\ppp$ is again principal. This proves (8)(b).  Finally, (8)(c) follows Proposition \ref{charideal} and statements (8)(a) and (8)(b).
\end{proof}

The {\it P\'olya-Ostrowski group of $D$} is defined for any Dedekind domain $D$ in \cite[Section II.3]{cah} and more generally for any Krull domain $D$ in \cite{cha}.   We generalize that definition as follows.

\begin{definition}
Let $D$ be an integral domain such that $\II_n(D)$ is $t$-invertible for all nonnegative integers $n$.  The {\it P\'olya-Ostrowski group $\PO(D)$ of $D$} is the subgroup of the $t$-class group $\Cl_t(D)$ generated by (the image in $\Cl_t(D)$ of) the $t$-invertible $t$-ideals $\II_n(D)$ for all $n$. 
\end{definition}  


With this definition, Theorem \ref{factprop} yields the following.

\begin{theorem}\label{trivcor}
Let $D$ be an integral domain such that $\II_n(D)$ is $t$-invertible for all $n$.  Then for any $f \in \Int_n(D)$, all of the coefficients of $f$ lie in $\II_n(D)$; one has $n!_D = (D[X]:_D \Int_n(D))$ and $\II_n(D) = n!_D^{-1}$; the P\'olya-Ostrowski group $\PO(D)$ is generated by the factorial ideals $n!_D$ for all $n$; and the following conditions are equivalent.
\begin{enumerate}
\item $\Int(D)$ has a regular basis.
\item $\II_n(D)$ is principal for every nonnegative integer $n$.
\item $n!_D$ is principal for every nonnegative integer $n$.
\item $\PO(D)$ is trivial.
\end{enumerate}
Moreover, if $D$ is polynomially L-regular, then $\Int(D_\ppp)$ has a regular basis for every prime ideal $\ppp$ of $D$ and the $D$-module $\Int(D)$ is locally free. 
\end{theorem}

Since every nonzero fractional ideal of a Krull domain is $t$-invertible, the above corollary generalizes the same result already known for Krull domains.  In Theorem \ref{equivthm}  of the next section we will show that $\II_n(D)$ is $t$-invertible for all $n$, and therefore the P\'olya-Ostrowski group $\PO(D)$ is defined, for a much larger class of domains $D$, including, for example, all TV PVMDs.


\section{The P\'olya-Ostrowski group}\label{sec:4}

In this section we use the results of the previous section to generalize the results of \cite[Section II.3]{cah} on Dedekind domains to a much larger class of domains, including, for example, all TV PVMDs.  For the remainder this paper we will be interested in the following conditions on an integral domain $D$.
\begin{enumerate}
\item[$(\mathcal{C}1)$] $D$ is polynomially L-regular.
\item[$(\mathcal{C}2)$] For any nonnegative integer $n$ there exist only finitely many $t$-maximal ideals $\ppp$ of $D$ with $N(\ppp) \leq n$.
\item[$(\mathcal{C}3)$] $\Pi_q$ is principal for every prime power $q$.\item[$(\mathcal{C}4)$] Every $t$-maximal ideal of $D$ of finite norm has finite height.
\item[$(\mathcal{C}5)$]  Every $t$-maximal ideal of $D$ of finite norm is $t$-invertible.
\end{enumerate}
Note that $(\mathcal{C}2)$ implies that $$\Pi_n = \prod_{{\ppp \in t\Max(D)} \atop {N(\ppp) = n}} \ppp,$$ and therefore $\Pi_n$ is the unique ideal of $D$ such that, for any prime ideal $\ppp$ of $D$, one has $\Pi_n D_\ppp = \ppp D_\ppp$ if $\ppp$ is $t$-maximal of norm $n$ and $\Pi_n D_\ppp = D_\ppp$ otherwise.

Following \cite[Chapter II]{cah}, for any nonnegative integer $n$ and any integer $k> 1$ we let 
$$w_k(n) = \sum_{i = 1}^\infty \left\lfloor \frac{n}{k^i} \right\rfloor.$$
Alternatively, by \cite[Exercise II.8 and Lemma II.2.4]{cah} one has
$$w_k(n) = \frac{n-s}{k-1} = \sum_{i = 1}^n v_k(i),$$
where $s$ is the sum of the digits of the $k$-adic expansion of $n$ and $v_k(i)$ for any positive integer $i$ is the largest nonnegative integer $t$ such that $k^t$ divides $i$.

\begin{lemma}\label{prinlemma}
Let $D$ be a local domain with principal maximal ideal $\ppp$ of finite norm $q$.  Then one has $\Pi_q = \ppp$ and $\Pi_n = D$ if $n \neq q$, and $\II_n = \left(\ppp^{w_q(n)}\right)^{-1} = \left(\ppp^{-1}\right)^{w_q(n)}$ for all $n$.  In particular, $\Int(D)$ has a regular basis.
\end{lemma}

\begin{proof}
This follows from \cite[Remark II.2.14]{cah} and the proof of \cite[Corollary II.2.9]{cah}.
\end{proof}

The following result generalizes the above lemma to nonlocal domains.

\begin{theorem}\label{regbasisalg}
Let $D$ be an integral domain satisfying conditions $(\mathcal{C}1)$, $(\mathcal{C}2)$, and $(\mathcal{C}3)$.  For every $n$ let $\pi_n \in D$ be a generator of $\Pi_n$.  Then $\sigma_n = \displaystyle \prod_{1 < k \leq n} \pi_k^{-w_k(n)}$ is a generator of $\II_n$ for all $n$, and therefore $\Int(D)$ has a regular basis.  For all $n$ and all $k > 1$, let $F_n = \frac{X^n-X}{\pi_n} \in \Int(D)$, and let $F_{k,n} = \prod_{i = 0}^r (F_k^{\circ i})^{n_i} \in \Int(D)$, where $n = n_0+ n_1k + \cdots + n_r k^r$ is the $k$-adic expansion of $n$.  Then $F_{k,n}$ has degree $n$ and leading coefficient $\pi_k^{-w_k(n)}$ for all $n,k$. For every integer $n > 1$ there exist $a_{2,n}, a_{3,n}, \ldots, a_{n,n} \in D$ such that
$$\sigma_n = \sum_{1< k \leq n} a_{k,n} \pi_k^{-w_k(n)}.$$
Let $G_0 = 1$, $G_1 = X$, and $\displaystyle G_n = \sum_{1 < k \leq n} a_{k,n} F_{k,n}$ for all $n > 1$.  Then $G_n \in \Int(D)$ has degree $n$ and leading coefficient $\sigma_n$ for every $n$, and therefore $\{G_0, G_1, G_2, \ldots\}$ is a regular basis of $\Int(D)$.
\end{theorem}

\begin{proof}
Let $$I = \prod_{1 < k \leq n} \pi_k^{-w_k(n)}D,$$ which is a fractional ideal of $D$.  Let $\ppp$ be a prime ideal of $D$.  Suppose first that $\ppp$ is $t$-maximal of finite norm $q$.  Then $\ppp D_\ppp = \pi_q D_\ppp$ is principal, and therefore by Theorem \ref{factprop}(3) and Lemma \ref{prinlemma} we have
$$\II_n(D)_\ppp = \II_n(D_\ppp) =  \left(\ppp^{w_q(n)} D_\ppp\right)^{-1}  = \pi_q^{-w_q(n)}D_\ppp = I D_\ppp.$$ 
(If $N(\ppp) > n$ then $\II_n(D)_\ppp = D_\ppp  = ID_\ppp$.)  If, on the other hand, $\ppp$ is not $t$-maximal of finite norm, then $\II_n(D)_\ppp = \II_n(D_\ppp) = D_\ppp = ID_\ppp$ by Theorem \ref{factprop}(4). Thus $\II_n(D) _\ppp = I D_\ppp$ for every prime ideal $\ppp$ of $D$, whence $\II_n(D)  = I$.    Finally, the remainder of the proposition follows exactly as in the proof of \cite[Propositions II.3.13 and II.3.14]{cah}.
\end{proof}


\begin{theorem}\label{equivthm}
Let $D$ be an integral domain satisfying conditions $(\mathcal{C}1)$, $(\mathcal{C}2)$, and $(\mathcal{C}5)$.
Then for every nonnegative integer $n$ the fractional ideals $\II_n$, $n!_D$, and $\Pi_n$ are $t$-invertible $t$-ideals, and one has $$\II_n = n!_D^{-1} = \prod_{{\ppp \in t\Max(D)} \atop {N(\ppp) \leq n}} \left(\ppp^{w_{N(\ppp)}(n)}\right)^{-1}$$
and $$n!_D = \prod_{{\ppp \in t\Max(D)} \atop {N(\ppp) \leq n}} \left( \ppp^{w_{N(\ppp)}(n)}\right)_t = \prod_{1 < q \leq n} \left(\Pi_q^{w_q(n)}\right)_t.$$
Moreover, the group $\PO(D)$ is generated by any of the following sets: $\{q!_D: q \textup{ is a prime power}\}$; $\{\II_q: q \textup{ is a prime power}\}$; and $\{\Pi_q: q \textup{ is a prime power}\}$.
In particular, the following conditions are equivalent.
\begin{enumerate}
\item $\Int(D)$ has a regular basis.
\item $\PO(D)$ is trivial.
\item $\II_n$ is principal for every nonnegative integer $n$.
\item $\II_q$ is principal for every prime power $q$.
\item $n!_D$ is principal for every nonnegative integer $n$.
\item $q!_D$ is principal for every prime power $q$.
\item $\Pi_n$ is principal for every nonnegative integer $n$.
\item $\Pi_q$ is principal for every prime power $q$.
\end{enumerate}
\end{theorem}

\begin{proof}
Let $$I = \prod_{{\ppp \in t\Max(D)} \atop {N(\ppp) \leq n}} \left(\ppp^{w_{N(\ppp)}(n)}\right)^{-1},$$ which is a well-defined fractional ideal of $D$.  Let $\ppp$ be a prime ideal of $D$.  If $\ppp$ is $t$-maximal of finite norm, hence $t$-invertible, then $\ppp D_\ppp$ is principal, and therefore by Theorem \ref{factprop}(3) and Lemmas \ref{prinlemma} and \ref{tinv} we have
$$\II_n(D)_\ppp = \II_n(D_\ppp) =  \left(\ppp^{w_{N(\ppp)}(n)} D_\ppp\right)^{-1}  = \left(\ppp^{w_{N(\ppp)}(n)}\right)^{-1} D_\ppp = I D_\ppp.$$
 If, on the other hand, $\ppp$ is not $t$-maximal of finite norm, then $\II_n(D)_\ppp = \II_n(D_\ppp) = D_\ppp = ID_\ppp$ by Theorem \ref{factprop}(4). Thus $\II_n(D) _\ppp = I D_\ppp$ for every prime ideal $\ppp$ of $D$, whence $\II_n(D)  = I$.  

Now $\Pi_n$ is a finite intersection, and product, of $t$-invertible $t$-maximal ideals and is therefore a $t$-invertible $t$-ideal.  If $J, J'$ are $t$-invertible fractional ideals, then so are $J^{-1}$ and $JJ'$, and one has $(J^{-1}J'^{-1})^{-1} = (JJ')_t$.  The given product for $\II_n(D)$, then, implies that $\II_n(D)$ is $t$-invertible, and therefore  by Theorem \ref{trivcor} one has $\II_n(D) = n!_D^{-1}$ and $\II_n(D)$ is a $v$-ideal, hence a $t$-ideal.  Moreover, $n!_D = \II_n(D)^{-1}$ is also a $t$-invertible $t$-ideal, and one has
$$n!_D = \left(\prod_{{\ppp \in t\Max(D)} \atop {N(\ppp) \leq n}} \left(\ppp^{w_{N(\ppp)}(n)}\right)^{-1} \right)^{-1} = \prod_{{\ppp \in t\Max(D)} \atop {N(\ppp) \leq n}} \left(\ppp^{w_{N(\ppp)}(n)}\right)_t$$
and therefore $$n!_D = \prod_{1 < q \leq n} \left(\Pi_q^{w_q(n)}\right)_t.$$ 
It follows that $\PO(D)$ is contained in the subgroup $G_1$ of $\Cl_t(D)$ generated by the image of $\Pi_q$ for all prime powers $q$.  Moreover, since $w_q(q) = 1$, one has
$$q!_D = \Pi_q \prod_{1 < q' < q} \left(\Pi_{q'}^{w_{q'}(q)}\right)_t,$$ so by induction on $q$ the image of $\Pi_q$ in $\Cl_t(D)$  for all $q$ is in the subgroup $G_2$ of $\PO(D)$ generated by $q!_D$ for all $q$, and therefore $G_1 \subseteq G_2$.  Therefore $\PO(D) \subseteq G_1 \subseteq G_2 \subseteq \PO(D)$, so equalities holds.  Moreover, $\PO(D) = G_2$ is also generated by $\II_q(D) = q!_D^{-1}$ for all $q$.   The equivalence of statements (1) through (8), then, follows from these equalities of groups and from Theorem \ref{trivcor}.
\end{proof}

\begin{remark}\label{polyafield}
Let $K$ be a finite Galois extension of $\QQ$.  Then for any prime power $q = p^r$, where $p$ is a prime, one has the following.
\begin{enumerate}
\item $\Pi_q = \sqrt{p \mathcal{O}_K}$ if $p$ is ramified in $K$ and $r$ equal to the inertial degree $f_p$.
\item $\Pi_q = \mathcal{O}_K$ if $p$ is ramified in $K$ and $r \neq f_p$.
\item $\Pi_q = p\mathcal{O}_K$ is principal if $p$ is unramified in $K$.
\end{enumerate}
In particular, $\PO(\mathcal{O}_K)$ is generated by $\sqrt{p\mathcal{O}_K}$ for the set of primes $p$ dividing the discriminant $\Delta_{K/\QQ}$, and $\Int(D)$ has a regular basis if and only if $\sqrt{p\mathcal{O}_K}$ is principal for all such $p$.
\end{remark}

Motivated by Theorem \ref{equivthm}, we find sufficient conditions for $(\mathcal{C}1)$,  $(\mathcal{C}2)$, and $(\mathcal{C}5)$ to hold.  A domain $D$ is said to be {\it of finite character} if every nonzero element of $D$ is contained in only finitely many maximal ideals of $D$.  Similarly, $D$ is said to be {\it of finite $t$-character} if every nonzero element of $D$ is contained in only finitely many $t$-maximal ideals of $D$.  Note, for example, that every TV domain is of finite $t$-character, and a PVMD of finite $t$-character is equivalently a domain of Krull type.

\begin{proposition}\label{normlemma}
Any domain $D$ of finite character or of finite $t$-character satisfies conditions $(\mathcal{C}1)$ and $(\mathcal{C}2)$.
\end{proposition}

\begin{proof}
 By \cite[Proposition 2.4]{ell2} any domain of finite character or finite $t$-character is polynomially L-regular.  To verify condition $(\mathcal{C}2)$ we may suppose that $D$ is infinite.  Let $q$ be a power of a prime.  Then there exists $a \in R$ with $a^q - a \neq 0$.  For every maximal ideal $\ppp$ of norm $q$ one has $a^q - a \in \ppp$, so since $D$ is of finite character or of finite $t$-character there are only finitely many such $\ppp$ that are $t$-maximal.   The lemma follows.
\end{proof}

\begin{proposition}\label{equivcor}
Let $D$ be an integral domain such that every $t$-maximal ideal of $D$ is $t$-invertible.  Then $D$ is of finite $t$-character if and only if every $t$-ideal $I$ of $D$ such that $I D_\ppp$ is principal for every $t$-maximal ideal $\ppp$ of $D$ is $t$-invertible.  For any such domain $D$, the all of the hypotheses (conditions $(\mathcal{C}1)$, $(\mathcal{C}2)$, and $(\mathcal{C}5)$) of Theorem \ref{equivthm} hold.
\end{proposition}

\begin{proof}
This follows from Proposition \ref{normlemma}, \cite[Proposition 2.4]{ell2},  and \cite[Corollary 2]{zaf}.
\end{proof}

\begin{corollary}
Let $D$ be an integral domain such that a $t$-ideal $I$ of $D$ is principal provided that $I$ is $t$-maximal or $ID_\ppp$ is principal for every $t$-maximal ideal $\ppp$ of $D$.  Then conditions $(\mathcal{C}1)$, $(\mathcal{C}2)$, $(\mathcal{C}3)$, and $(\mathcal{C}5)$ hold; in particular, statements (1) through (8) of Theorem \ref{equivthm} hold, and $\Int(D)$ has a regular basis.
\end{corollary}

\begin{proof}
By Proposition \ref{equivcor}, $D$ is a domain of finite $t$-character such that every $t$-maximal ideal of $D$ is $t$-invertible.  Moreover, every $t$-invertible $t$-ideal of $D$ is principal, hence $\Cl_t(D)$ is trivial, so $\PO(D)$ is trivial.  The result therefore follows from Proposition \ref{equivcor}.
\end{proof}

An {\em H domain} is a domain in which every $v$-invertible ideal is $t$-invertible.  Every TV domain is an H domain of finite $t$-character; we do not know if the converse holds, even for PVMDs.  By the following corollary, Theorem \ref{equivthm} and Proposition \ref{equivcor} apply in particular to any H PVMD of finite $t$-character, hence to any TV PVMD.

\begin{corollary}\label{pvmdcor}  Let $D$ be a PVMD.  Then $D$ is an H domain if and only if every $t$-maximal ideal of $D$ is $t$-invertible.  Also, $D$ is of finite $t$-character if and only if every $t$-ideal $I$ of $D$ such that $I D_\ppp$ is principal for every $t$-maximal ideal $\ppp$ of $D$ is $t$-invertible.  If $D$ is an H PVMD of finite $t$-character (or if $D$ is a TV PVMD), then all of the hypotheses (conditions $(\mathcal{C}1)$, $(\mathcal{C}2)$, and $(\mathcal{C}5)$) of Theorem \ref{equivthm} hold. 
\end{corollary}

\begin{proof}
This follows from \cite[Proposition 1.5]{ell2} and \cite[Proposition 5]{zaf}.
\end{proof}





\section{The global case}\label{sec:5}

In this section we extend Proposition \ref{presentation} to nonlocal domains. For the remainder of this paper we will use the following notation.

\begin{definition} Let $\mathcal{C}$ denote the class of integral domains $D$ satisfying the following four conditions.
\begin{enumerate}
\item[$(\mathcal{C}1)$] $D$ is polynomially L-regular.
\item[$(\mathcal{C}2)$] For any nonnegative integer $n$ there exist only finitely many $t$-maximal ideals $\ppp$ of $D$ with $N(\ppp) \leq n$.
\item[$(\mathcal{C}3)$] $\Pi_q$ is principal for every prime power $q$.\item[$(\mathcal{C}4)$] Every $t$-maximal ideal of $D$ of finite norm has finite height.
\end{enumerate}
\end{definition}

\begin{remark} \
\begin{enumerate}
\item By Theorem \ref{regbasisalg} $\Int(D)$ has a regular basis for any domain $D$ in the class $\mathcal{C}$.
\item Any finite dimensional local domain with principal maximal ideal is in the class $\mathcal{C}$.  More generally, a local domain is in the class $\mathcal{C}$ if and only if its maximal ideal has infinite norm  or else is principal and has finite height and norm.
\item A Krull domain $D$ is in the class $\mathcal{C}$ if and only if $\Int(D)$ has a regular basis.
\item Any H PVMD of finite $t$-character (or any TV PVMD) such that $\Int(D)$ has a regular basis satisfies conditions $(\mathcal{C}1)$, $(\mathcal{C}2)$, and $(\mathcal{C}3)$.
\item Any domain $D$ of finite character or of finite $t$-character satisfies conditions $(\mathcal{C}1)$ and $(\mathcal{C}2)$.
\end{enumerate}
\end{remark}

For any domain $D$ in the class $\mathcal{C}$, the $D$-algebra $\Int(D)$ has a presentation by generators and relations as in the following theorem.

\begin{proposition}\label{polyaprop}
Let $D$ be an integral domain in the class $\mathcal{C}$.  For each $q$ let $\pi_q$ be a generator of $\Pi_q$ and let $F_q = \frac{X^q - X}{\pi_q} \in \Int(D)$.   Then the unique $D$-algebra homomorphism
$$D[\{X, X_{q,k}  :  q \textup{ is a prime power and } k \in \ZZ_{\geq 0}\}] \longrightarrow \Int(D)$$
$$X \longmapsto X$$
$$X_{q,k} \longmapsto F_q^{\circ k}$$
is surjective, and its kernel is equal to the ideal $J$ generated by $X_{q,0} - X$ and $X_{q,k}^q - X_{q,k} - \pi_q X_{q,{k+1}}$ for all $q,k$.
\end{proposition}

\begin{proof}
Let $\varphi$ denote the given $D$-algebra homomorphism, and let $$A = D[\{X, X_{q,k} : q \textup{ is a prime power and } k \in \ZZ_{\geq 0}]/J.$$  The homomorphism $\varphi$ induces
a $D$-algebra homomorphism $$\psi: A \longrightarrow \Int(D).$$  We show that $\psi$ is an isomorphism.  Let $\ppp$ be a prime ideal of $D$.  If $\ppp$ is not $t$-maximal of finite norm, then $\pi_q$ is a unit in $D_\ppp$ for all $q$  since $\Pi_{q} \nsubseteq \ppp D_\ppp$, and $A_\ppp \cong D_\ppp[X]$. Thus, by \cite[Lemma 2.2]{ell2}, the localization $\psi_{\ppp}$ of $\psi$ at $\ppp$ is the isomorphism
$$A_\ppp \longrightarrow \Int(D)_\ppp = \Int(D_\ppp) = D_\ppp[X],$$
Suppose, on the other hand, that $\ppp$ is $t$-maximal of finite norm, say $|D/\ppp| = q$.  Then $\pi_q D_\ppp = \ppp D_\ppp$, and $\pi_{q^\prime}$ is a unit in $D_\ppp$ for all prime powers $q^\prime \neq q$.  It follows, then, that $$A_\ppp \cong D_\ppp[X_0, X_1, X_2, \ldots]/I,$$ where $I$ is defined as in Proposition
\ref{presentation}.  Moreover, $\psi_\ppp$ is the same as the homomorphism
$$D_\ppp[X_0, X_1, X_2, \ldots]/I \longrightarrow \Int(D)_\ppp = \Int(D_\ppp)$$
$$X_k \longmapsto F_q^{\circ k},$$
of Proposition \ref{presentation}.  By that proposition, then, $\psi_\ppp$ is an isomorphism.  Therefore $\psi_\ppp$ is an isomorphism for all prime ideals $\ppp$ of $D$, so $\psi$ is an isomorphism.
\end{proof}


Now, let $D$ be a domain and $\XX$ a set.  There exists a unique $D$-algebra homomorphism 
$$\theta_\XX: \bigotimes_{X \in \XX} \Int(D) \longrightarrow \Int(D^\XX)$$ sending $X \in \Int(D)$ to $X \in \Int(D^\XX)$ for all $X \in \XX$,
where the tensor product is over $D$.  By \cite[Proposition 6.8(a) and 6.10(d)]{jess}, the map $\theta_\XX$ is an isomorphism if $\Int(D)$ is free as a $D$-module or if $D$ is polynomially L-regular and $\Int(D)$ is locally free as a $D$-module.  The isomorphisms $\theta_\XX$ allow us to extend Propositions \ref{presentation} and \ref{polyaprop} to multivariate integer-valued polynomials, as follows.

\begin{proposition}\label{presentation2}
Let $D$ be a local integral domain with principal maximal ideal $\pi D$ and finite residue field of order $q$.  Let $F_q = \frac{X^q - X}{\pi} \in \Int(D)$.  Let $\XX = \{X_i\}_{i \in I}$ be a set of variables.   The unique $D$-algebra homomorphism
$$\varphi: {\begin{array}{rrr} D[\{X_{i,k} : i \in I, k \in \ZZ_{\geq 0}\}] & \longrightarrow & \Int(D^\XX) \\
X_{i,k}  &\longmapsto & F_q^{\circ k}(X_i)
\end{array}}$$
is surjective, and if $D$ has finite Krull dimension then $\ker \varphi$ is equal to the ideal generated by $X_{i,k}^q - X_{i,k} - \pi X_{i,k+1}$ for
all $i,k$.
\end{proposition}

\begin{theorem}\label{presentation3}
Let $D$ be an integral domain in the class $\mathcal{C}$.   
For each $q$ let $\pi_q$ be a generator of $\Pi_q$ and let $F_q = \frac{X^q - X}{\pi_q} \in \Int(D)$.  Let $\XX = \{X_i\}_{i \in I}$ be a set of variables.  The unique $D$-algebra homomorphism
$$D[\{X_i, X_{i,q,k} : i \in I, q \textup{ is a prime power, and } k \in \ZZ_{\geq 0}\}] \longrightarrow \Int(D^\XX)$$
$$X_i \longmapsto X_i$$
$$X_{i,q,k} \longmapsto F_q^{\circ k}(X_i)$$
is surjective, and its kernel is equal to the ideal generated by $X_{i,q,0} - X_i$ and $X_{i,q,k}^q - X_{i,q,k} - \pi_q X_{i,q,{k+1}}$ for all $i,q,k$.
\end{theorem}

\section{Quotients of integer-valued polynomial rings}\label{sec:6}

A ring $A$ is said to be {\it binomial} if $A$ is $\ZZ$-torsion-free and is closed under the operation $x \longmapsto \frac{x(x-1)(x-2)\cdots(x-n+1)}{n!}$ on $A \otimes_\ZZ \QQ$ for every positive integer $n$.   For example, any $\QQ$-algebra is binomial; any localization or completion of $\ZZ$ is binomial; and the domain $\Int(D^\XX)$ is binomial for any binomial domain $D$ and any set $\XX$. Binomial rings were introduced by Philip Hall in his groundbreaking work \cite{hal} on nilpotent groups.  Hall proved the existence of an action of any binomial ring on a class of nilpotent groups, generalizing exponentiation of elements of an abelian group by the integers and analogous to exponentiation of elements of a uniquely divisible group by the rational numbers.  In \cite{ell7}, it is shown that a binomial ring is equivalently (1) a $\lambda$-ring on which all Adams operations are the identity; (2) a $\ZZ$-torsion-free ring $A$ such that the Frobenius endomorphism of $A/pA$ is the identity for every prime number $p$; (3) a $\ZZ$-torsion-free ring isomorphic to a quotient of a (possibly infinite) tensor power of $\Int(\ZZ)$; and (4) a $\ZZ$-torsion-free ring isomorphic to a quotient of $\Int(\ZZ^\XX)$ for some set $\XX$.

In this section and the next we generalize the equivalences (2) through (4) above to domains more general than $\ZZ$.  Let $D$ be an integral domain and $A$ a $D$-algebra.   Following \cite{ell}, we say that $A$ is {\it weakly polynomially complete}, or {\it WPC}, if for every $a \in A$ there is a $D$-algebra homomorphism $\Int(D) \longrightarrow A$ sending $X$ to $a$.   Any quotient of a WPC $D$-algebra is WPC.   If $A$ is a domain extension of $D$, then $A$ is WPC if and only if $\Int(D) \subseteq \Int(A)$. A $\ZZ$-torsion-free ring $A$ is a WPC $\ZZ$-algebra if and only if $A$ is a binomial; and for any number field $K$, the localization $S^{-1}\mathcal{O}_K$ of $\mathcal{O}_K$ at the multiplicative subset $S$ of $\ZZ$ generated by the set of prime numbers $p$ that do not split completely in $\mathcal{O}_K$ is the smallest WPC extension of $\ZZ$ containing $\mathcal{O}_K$ \cite[Example 7.3(3)]{jess2}.  Characterizations of the divisorial (or flat) weakly polynomially complete extensions of any Krull domain are given in \cite[Theorem 1.2]{ell2}.   We remark that, if $D$ is a principal ideal domain with finite residue fields, then $\Int(D)$ left-represents a right adjoint for the inclusion functor from $D$-torsion-free WPC $D$-algebras to $D$-algebras \cite[Theorem 1.6]{ell}.

The problem we consider is to characterize the WPC $D$-algebras if $D$ is a Krull domain (or Dedekind domain).  Let us say that $A$ is {\it almost polynomially complete}, or {\it APC}, if for every set $\XX$ and for any $(a_X)_{X \in \XX} \in A^\XX$ there exists a $D$-algebra homomorphism $\Int(D^\XX) \longrightarrow A$ sending $X$ to $a_X$ for all $X \in \XX$.   In other words, $A$ is APC if and only if $A$ is isomorphic as a $D$-algebra to a quotient of $\Int(D^\XX)$ for some set $\XX$. By \cite[Propositions 7.4 and 7.7]{jess}, if $A$ is a domain extension of $D$, then $A$ is APC if and only if $A$ is an {\it almost polynomially complete extension} of $D$ in the sense of \cite[Section 7]{jess}, that is, $\Int(D^n) \subseteq \Int(A^n)$ for all positive integers $n$.  Any quotient of an APC $D$-algebra is APC.  Clearly any APC $D$-algebra is WPC; we suspect that the converse does not hold but do not know a counterexample.  However, by \cite[Theorem 3.11]{ell2} and the universal property of tensor products, we have the following.

\begin{proposition}
Let $D$ be an integral domain and $\XX$ a set.  There exists a unique $D$-algebra homomorphism 
$$\theta_\XX: \bigotimes_{X \in \XX} \Int(D) \longrightarrow \Int(D^\XX)$$ sending $X \in \Int(D)$ to $X \in \Int(D^\XX)$ for all $X \in \XX$,
where the tensor product is over $D$.   If $\theta_\XX$ is an isomorphism for all finite sets $\XX$, which holds, for example, if $\Int(D)$ is free as a $D$-module or if $D$ is polynomially L-regular and $\Int(D)$ is locally free as a $D$-module, 
then $\theta_\XX$ is an isomorphism for all sets $\XX$, and a $D$-algebra $A$ is WPC if and only if it is APC.
\end{proposition}

As in \cite[Section 1]{jess}, we say that a domain $A$ containing $D$ is a {\it polynomially complete extension of $D$} if $D$ is a polynomially dense subset of $A$. By \cite[Proposition 7.2]{jess}, every polynomially complete extension of $D$ is APC, and if $D$ is infinite then a $D$-algebra $A$ is APC if and only if $A$ is isomorphic as a $D$-algebra to a quotient of some polynomially complete extension of $D$.  On the other hand, if $D$ is finite, or more generally if $\Int(D) = D[X]$, then by \cite[Lemma 7.1]{jess} every $D$-algebra is APC.

The WPC and polynomially complete extensions of a domain are studied, for example, in \cite[Sections 5 and 6]{cah1}, \cite[Section IV.3]{cah}, \cite{ell2}, \cite{ger}, and \cite{jess}.

The following result as a special case characterizes the WPC algebras over a discrete valuation domain.

\begin{proposition}\label{dvd}
Let $D$ be a finite dimensional local domain with principal maximal ideal $\pi D$ and finite residue field of order
$q$.  Let $A$ be a $D$-algebra.  Then $A$ is WPC if and only if $a^q \equiv a \ (\mod \pi A)$ for all $a \in A$.
\end{proposition}

\begin{proof}
Suppose first that $A$ is WPC.  Let $a \in A$, and let $\varphi: \Int(D) \longrightarrow A$ be a $D$-algebra homomorphism sending $X$ to $a$.  Since the polynomial $f = \frac{X^q - X}{\pi}$ lies in $\Int(D)$, one has $a^q - a  = \varphi(X^q - X) = \pi \varphi(f) \in \pi A$, so $a^q \equiv a \ (\mod \pi A)$.

Conversely, suppose that $a^q \equiv a \ (\mod \pi A)$ for all $a \in A$.  
Let  $a \in A$.  Define an infinite sequence $a_0, a_1, a_2, \ldots$ recursively as follows.  Let $a_0 = a$, and
let $a_{k+1}$ be any element of $A$ such that $a_{k}^q - a_{k} = \pi a_{k+1}$.  Consider the unique $D$-algebra homomorphism
$$\varphi: {\begin{array}{rrr} D[X_0, X_1, X_2, \ldots] & \longrightarrow & A \\
 X_{k} & \longmapsto & a_k.
\end{array}}$$
The polynomials $X_{k}^q - X_{k} - \pi X_{k+1}$ lie in $\ker \varphi$ for all $k$.  By  Proposition \ref{presentation}, therefore, $\varphi$ induces a $D$-algebra homomorphism from $\Int(D)$ into $A$ sending $X$ to $a$.  Thus $A$ is WPC.
\end{proof}

Next, let $\mathcal{D}$ be the class of domains $D$ such that $\Pi_q = \pi_q D$ is principal for all $q$ and $\Int(D)$ has a $D$-algebra presentation as in Proposition \ref{polyaprop}.  Then $\mathcal{C} \subseteq \mathcal{D}$ by Proposition \ref{polyaprop}.  We do not know if the domains $D$ satisfying conditions $(\mathcal{C}1)$, $(\mathcal{C}2)$, and $(\mathcal{C}3)$ are  in the class $\mathcal{D}$.

\begin{theorem}\label{polyathm}
For any integral domain $D$ in the class $\mathcal{D}$ (or in the class $\mathcal{C}$), the following conditions are equivalent.
\begin{enumerate}
\item $A$ is a WPC $D$-algebra.
\item $a^{N(\ppp)} \equiv a \ (\mod \ppp A)$ for all $a \in A$ and for every $t$-maximal ideal $\ppp$ of $D$ of finite norm $N(\ppp)$.
\item $a^q \equiv a \ (\mod \Pi_q A)$ for all $a \in A$ and for every prime power $q$.
\end{enumerate}
Moreover, if $D$ is in the class $\mathcal{C}$, then the above conditions are equivalent to the following.
\begin{enumerate}
\item[(4)] $A$ is isomorphic as a $D$-algebra to a quotient of $\Int(D^\XX)$ for some set $\XX$.
\end{enumerate}
\end{theorem}

\begin{proof}
The proof is a straightforward extension of the proof of Proposition \ref{dvd}.
\end{proof}

\section{Local versus global behavior}\label{sec:7}

In this section we investigate the local/global behavior of WPC algebras, and we characterize the ``locally WPC'' algebras over any Krull domain.


\begin{lemma}\label{lreg}
Let $D$ be an integral domain with quotient field $K$ and $A$ a $D$-torsion-free $D$-algebra.
\begin{enumerate}
\item $A$ is a WPC $D$-algebra if and only if for every $f \in \Int(D)$ one has
$f(a) \in A \subseteq A \otimes_D K$ for every $a \in A$.
\item If $A$ is a WPC $D$-algebra and $D$ is polynomially L-regular, then $S^{-1}A$ is a WPC $S^{-1}D$-algebra for every multiplicative subset $S$ of $D$.
\item Suppose that $\mathcal{P}$ is a set of prime ideals of $D$ such that $A = \bigcap_{\ppp \in \mathcal{P}} A_\ppp$.   If $A_\ppp$ is a WPC $D_\ppp$-algebra for all $\ppp \in \mathcal{P}$, then $A$ is a WPC $D$-algebra.
\end{enumerate}
\end{lemma}

\begin{proof}
This is clear.
\end{proof}

\begin{lemma}\label{extext}
Let $D$ be an integral domain, $D'$ an extension $D$, and $A$ a $D'$-algebra.
\begin{enumerate}
\item If $A$ is a WPC $D$-algebra, $D'$ is flat over $D$, and $D$ is polynomially L-regular, then $A$ is a WPC $D'$-algebra. 
\item If $A$ is a WPC $D'$-algebra and $D'$ is a WPC $D$-algebra, then $A$ is a WPC $D$-algebra.
\end{enumerate}
\end{lemma}

\begin{proof}
To prove (1) let $a \in A$.  By hypothesis there is a $D$-algebra homomorphism $\Int(D) \longrightarrow A$ sending $X$ to $a$.  By \cite[Proposition 2.3]{ell2}, then, there is a $D'$-algebra homomorphism $$\Int(D') = D' \Int(D) \cong D' \otimes_D \Int(D) \longrightarrow D' \otimes_D A \longrightarrow A$$
sending $X \in \Int(D')$ to $a$.  Therefore $A$ is a WPC $D'$-algebra.

To prove (2), again let $a \in A$.  By hypothesis there is a $D$-algebra homomorphism
$$\Int(D) \subseteq \Int(D') \longrightarrow A$$
sending $X \in \Int(D)$ to $a$.  Therefore $A$ is a WPC $D$-algebra.
\end{proof}

Because any localization of a domain $D$ at a multiplicative subset is a flat WPC $D$-algebra, we have the following corollary. 

\begin{corollary}\label{extextcor}
Let $D$ be a polynomially L-regular integral domain and $A$ a $D$-algebra.  Let $S$ be a multiplicative subset of $D$.   Then $S^{-1}A$ is a WPC $D$-algebra if and only if $S^{-1}A$ is a WPC $S^{-1}D$-algebra.   
\end{corollary}

Let us say that a $D$-algebra $A$ is {\it locally WPC} if $A_\ppp$ is a WPC $D_\ppp$-algebra for every prime ideal $\ppp$ of $D$.  If $D$ is polynomially L-regular, then by Corollary \ref{extextcor} this holds if and only if $A_\ppp$ is a WPC $D$-algebra for every prime ideal $\ppp$ of $D$.

\begin{lemma}\label{localprop}
The following conditions are equivalent for any integral domain $D$ and any $D$-algebra $A$.
\begin{enumerate}
\item $A$ is locally WPC.
\item $A_\ppp$ is a WPC $D_\ppp$-algebra for every maximal ideal $\ppp$ of $D$.
\item $A_\ppp$ is a WPC $D_\ppp$-algebra for every $t$-maximal ideal $\ppp$ of $D$ of finite norm.
\end{enumerate}
\end{lemma}

\begin{proof}
 If $\ppp$ is a prime ideal of $D$ that is not $t$-maximal of finite norm, then by \cite[Lemma 2.2]{ell2} one has $\Int(D_\ppp) = D_\ppp[X]$, and therefore $A_\ppp$ is a WPC $D_\ppp$-algebra.  Therefore (3) implies (1) and so the three conditions are equivalent.
\end{proof}

Any domain $D$ in the class $\mathcal{C}$ satisfies the hypothesis of the following theorem.

\begin{theorem}\label{polybthm}
Let $D$ be an integral domain and $A$ a $D$-algebra.  Suppose that $\ppp D_\ppp$ is principal and has finite height for every $t$-maximal ideal $\ppp$ of $D$ of finite norm.  Then $A$ is locally WPC if and only if, for every $t$-maximal ideal $\ppp$ of $D$ of finite norm $q = N(\ppp)$, any of the following equivalent conditions holds.
\begin{enumerate}
\item $a^{N(\ppp)} \equiv a \ (\mod \, \ppp A)$ for all $a \in A$.
\item The endomorphism $a \longmapsto a^q$ of $A/\ppp A$ is the identity.
\item $A/\ppp A$ is locally isomorphic to $D/\ppp$ as a $D$-algebra.
\item $A/\ppp A$ is reduced, and every residue field of $A/\ppp A$ is isomorphic to $D/\ppp$ as a $D$-algebra.
\item $A/\ppp A$ is isomorphic to a subring of $\FF_q^Y$ for some set $Y$.
\item For every maximal ideal $\MM$ of $A$ lying over $\ppp$, one has $\ppp A_\MM = \MM A_\MM$ and $A/\MM \cong D/\ppp$ as $D$-algebras.
\item For every prime ideal $\PPP$ of $A$ lying over $\ppp$, one has $\ppp A_\PPP = \PPP A_\PPP$ and $A/\PPP \cong D/\ppp$ as $D$-algebras.
\end{enumerate}
Moreover, if $A/\ppp A$ is semi-local or Noetherian, then each of the above conditions is equivalent to the following.
\begin{enumerate}
\item[(8)] $\ppp A = \MM_1 \MM_2 \cdots \MM_r$ for distinct maximal ideals
$\MM_i$ of $A$ such that $A/\MM_i \cong D/\ppp$.
\end{enumerate}
\end{theorem}

\begin{proof}
By \cite[Propositions 4.1 and 4.2]{jess}, we need only show that $A$ is locally WPC if and only if condition (1) holds for all $t$-maximal ideals $\ppp$ of $D$ of finite norm.

Suppose first that $A$ is locally WPC, so $A_\ppp$ is a WPC $D_\ppp$-algebra for every prime ideal $\ppp$ of $D$.  Let $\ppp$ be a $t$-maximal ideal of $D$ of finite norm $q = N(\ppp)$, and let $a \in A$.  Since $q = |D_\ppp/\ppp D_\ppp|$, by Proposition \ref{dvd} one has $a^q \equiv a \ (\mod \, \ppp A_\ppp)$.  Therefore $a^q - a \in \ppp A_\ppp$, so there exists $u_\ppp \in D\backslash \ppp$ so that $u_\ppp (a^q -a) \in \ppp A$.  Let $v_\ppp \in D \backslash \ppp$ with $v_\ppp u_\ppp \equiv 1 \ (\mod \, \ppp)$.  Then $$a^q -a \equiv v_\ppp u_\ppp(a^q  - a) \equiv 0 \ (\mod \, \ppp A),$$
whence $a^q \equiv a \ (\mod \, \ppp A)$. 

Suppose, conversely, that $a^{N(\ppp)} \equiv a \ (\mod \, \ppp A)$ for all $a \in A$ and for every $t$-maximal ideal $\ppp$ of $D$ of finite norm.  Then for all $a \in A$ and  $u \in D \backslash \ppp$  one has$$a^{N(\ppp)}u \equiv au^{N(\ppp)} \ (\mod \, \ppp A_\ppp)$$ and therefore $$(a/u)^{N(\ppp)} \equiv a^{N(\ppp)}/u^{N(\ppp)} \equiv a/u \ (\mod \, \ppp A_\ppp).$$  Therefore $x^q \equiv x \ (\mod \, \ppp A_\ppp)$ for all $x \in A_\ppp$, where $q = |D_\ppp/\ppp D_\ppp|$, so $A_\ppp$ is a WPC $D_\ppp$-algebra by Proposition \ref{dvd}. Thus $A$ is locally WPC by Lemma \ref{localprop}.
\end{proof}


\begin{corollary}\label{polybcor}
Let $D$ be an integral domain and $A$ a $D$-algebra.  If $D$ is in the class $\mathcal{C}$, or if $A$ is $D$-torsion-free, then $A$ is WPC if and only if $A$ is locally WPC.
\end{corollary}

\begin{corollary}\label{polybcor1}
Let $D$ be a Krull domain and $A$ a $D$-algebra.   Then $A$ is locally WPC if and only if $a^{N(\ppp)} \equiv a \ (\mod \, \ppp A)$ for every height one prime ideal $\ppp$ of $D$ of finite norm.
\end{corollary}

Corollaries \ref{polybcor} and \ref{polybcor1} motivate the following problem.

\begin{problem}\label{dedekindconjecture}
Let $D$ be a Dedekind domain and $A$ a $D$-algebra.  Is it true that $A$ is WPC if and only if $A$ is locally WPC?  Equivalently, is it true that $A$ is WPC if and only if for every height one prime ideal $\ppp$ of $D$ of finite norm $N(\ppp)$ one has $a^{N(\ppp)} \equiv a \ (\mod \ppp A)$ for all $a \in A$?  If so, then does the equivalence hold more generally if $D$ is a Krull domain?   If not, then which, if either, implication is true?  
\end{problem}

Corollary \ref{polybcor}  shows that the answer to the above problem is affirmative under the added hypothesis that $\Int(D)$ has a regular basis or $A$ is $D$-torsion-free.

\section{Weak polynomial completions}\label{sec:8}

Let $D$ be an integral domain with quotient field $K$.  For any domain extension $A$ of $D$, there is a smallest WPC extension of $D$ containing $A$ and contained in $A \otimes_D K$, called the {\it weak polynomial completion of $A$} and denoted $w_D(A)$ \cite[Section 8]{jess}.

An example of weak polynomial completion is as follows.  Let $A$ be a domain extension of $D$ and $\alpha \in A$.  Following \cite[Section 5]{cah1} we let $D_\alpha = \{f(\alpha): f \in \Int(D)\}$ denote the {\it ring of values of $\Int(D)$ at $\alpha$}, which is a $D$-subalgebra of $A$ containing $D[\alpha] = \{f(\alpha): f \in D[X]\}$.  An easy argument shows that $D_\alpha = w_D(D[\alpha])$.  By \cite[Example 7.3]{ell2}, if $D = \ZZ[T]$, then $\Int(D) = D[X]$ and therefore any extension $A$ of $D$ is WPC, but if $A = \ZZ[T/2]$ then $A$ is not a polynomially complete extension of $D$.  In particular, $A = D_{T/2}$ is not a polynomially complete extension of $D$.  This provides a negative answer to \cite[Question 5.8]{cah1}.  

\begin{theorem}\label{numthm}
Let $D$ be a Dedekind domain with quotient field $K$, and let $L$ be a finite Galois extension of $K$ with ring of integers $D'$.  Let $S$ denote the complement of the union of the prime ideals of $D$ that split completely in the Dedekind domain $D'$, and let $D'' = w_D(D')$.
\begin{enumerate}
\item If $\ppp$ is a prime ideal of $D$ with $\ppp D'' \neq D''$, then $\ppp$ splits completely in $D'$.
\item $D'' \supseteq S^{-1}D'$.
\item $D'' = S^{-1}D'$ if the class group of $D$ is torsion.
\end{enumerate}
\end{theorem}

\begin{proof}
First we prove (1). 
Suppose that $\ppp D'' \neq D''$, so $\ppp D'' \subseteq \ppp''$ for some maximal ideal $\ppp''$ of $D''$.     Let $\ppp' = D' \cap \ppp''$.  Note that $D''$, being an overring of the Dedekind domain $D'$, is a Dedekind domain and is flat over $D'$, and one has $D''_{\ppp''} = D'_{\ppp'}$ and therefore $\ppp'' D_{\ppp''} = \ppp' D'_{\ppp'}$.  Therefore, by Theorem \ref{polybthm}
one has 
$$\ppp D'_{\ppp'} = \ppp D''_{\ppp''} = \ppp'' D''_{\ppp''} = \ppp' D'_{\ppp'}$$ and $D''/\ppp'' \cong D'/\ppp' \cong D/\ppp$ as $D$-algebras.  Therefore $e_{\ppp' | \ppp} = f_{\ppp' | \ppp} = 1$, so since $L$ is Galois over $K$ it follows that $\ppp$ splits completely in $D'$.

Now, let $x \in D \backslash S$, so $x$ is not in any prime ideal of $D$ that splits completely in $D'$.  Writing $xD = \ppp_1^{k_1} \cdots \ppp_r^{k_r}$ with each $\ppp_i \subseteq D$ prime, we see from (1) that $\ppp_i D'' = D''$ for each $i$, and therefore $xD'' = D''$, that is, $x$ is a unit in $D''$.  Therefore $S^{-1}D' \subseteq D''$.  This proves (2).

Finally, suppose that the class group of $D$ is torsion.  To show that $D'' = S^{-1}D'$, by (2) it suffices to show that $S^{-1}D'$ is a WPC extension of $D$.  Let $\ppp$ be a prime ideal of $D$ with finite residue field, and let $\ppp'$ be any prime ideal of $S^{-1}D'$ lying over $\ppp$.
Then $\ppp'$ doesn't intersect $S$, so $\ppp = D \cap \ppp'$ is contained in the complement of $S$, which is equal to the union of the prime ideals of $D$ that split completely in $D'$.  Since $\ppp^k$ is principal for some $k$, say, $\ppp = (a)$, it follows that $a \in D \backslash S$, so $\ppp^k \subseteq \qqq$ for some prime $\qqq$ that splits completely in $D'$.  Therefore $\ppp = \qqq$ splits completely in $D'$.  Thus $\ppp D' = \qqq_1 \qqq_2 \cdots \qqq_r$, where the $\qqq_i$ are distinct prime ideals of $D'$ for which $D'/\qqq_i \cong D/\ppp$.  Reindexing the $\qqq_i$, we may assume there is a nonnegative integer $s$ such that $\qqq_i$ meets $S$ if and only if $i > s$.  It follows, then, that $$\ppp S^{-1}D' = (\qqq_1 S^{-1}D')(\qqq_2 S^{-1}D') \cdots (\qqq_s S^{-1}D'),$$ where the $\qqq_i S^{-1}D'$ are distinct prime ideals of $S^{-1}D'$ for which $$S^{-1}D'/\qqq_i S^{-1}D' \cong S^{-1}(D/\qqq_i) \cong S^{-1} (D/\ppp) \cong D/\ppp.$$ Therefore $S^{-1}D'$ is a locally WPC extension of $D$ by Theorem \ref{polybthm}, hence a WPC extension of $D$ since $S^{-1}D'$ is $D$-torsion-free.
\end{proof}

\begin{corollary}
Let $D$ be a Dedekind domain with torsion class group and with quotient field $K$, and let $L$ be a finite Galois extension of $K$ with ring of integers $D'$.  Let $S$ denote the complement of the union of the prime ideals of $D$ that split completely in the Dedekind domain $D'$, and let $D''$ be an overring of $D'$. Then $D''$ is a WPC extension of $D$ containing $D'$ if and only if $D'' \supseteq S^{-1}D'$.
\end{corollary}


\begin{thebibliography}{}



\bibitem{ati} M.\ F.\ Atiyah and I.\ G.\ MacDonald, {\it Introduction to Commutative Algebra}, Addison-Wesley Publishing Company, New York, 1969.

\bibitem{cah1} P.-J.\ Cahen, Polynomial closure, J.\ Number Theory 61 (1996) 226--247.

\bibitem{cah} P.-J.\ Cahen and J.-L.\ Chabert, {\it Integer-Valued
Polynomials},  Mathematical Surveys and Monographs, Volume 48,
American Mathematical Society, 1997.

\bibitem{cha} J.-L.\ Chabert, Factorial groups and P\'olya groups in Galoisian extension of $\QQ$, in: {\it Commutative Ring Theory and Applications: Proceedings of the Fourth International Conference},  Eds.\ Fontana, Kabbaj, and Wiegand, Lecture Notes in Pure and Applied Mathematics, Volume 231, Marcel Dekker, Inc., New York, 2003.

\bibitem{cha2} J.-L.\ Chabert, Generalized factorial ideals, Arabian J.\ Sc.\ and Eng.\ 26 (2001) 51--68.

\bibitem{des} E.\ de Shalit and E.\ Iceland, Integer-valued polynomials and Lubin-Tate formal groups, J.\ Number Theory 129 (3) (2009) 632--639.

\bibitem{ell} J.\ Elliott, Biring and plethory structures on integer-valued polynomial rings, submitted.

\bibitem{ell2} J.\ Elliott, Integer-valued polynomial rings, $t$-closure, and associated primes, http://arxiv.org/abs/1105.0142, to appear in Comm.\ Algebra.

\bibitem{ell3} J.\ Elliott, Some new approaches to integer-valued polynomial rings, in: {\it Commutative Algebra and its Applications: Proceedings of the Fifth Interational Fez Conference on Commutative Algebra and Applications}, Eds. Fontana, Kabbaj, Olberding, and Swanson, de Gruyter, New York, 2009.

\bibitem{ell7} J. Elliott, Binomial rings, integer-valued polynomials, and $\lambda$-rings, J.\ Pure Appl.\ Alg.\ 207 (2006) 165--185.

\bibitem{jess} J.\ Elliott, Universal properties of integer-valued polynomial rings, J.\ Algebra 318 (2007) 68--92.

\bibitem{jess2} J.\ Elliott, Binomial rings, integer-valued polynomials, and $\lambda$-rings, J.\ Pure Appl.\ Alg.\ 207 (2006) 165--185.

\bibitem{gil} R.\ Gilmer, {\it Multiplicative Ideal Theory}, Marcel Dekker, Inc., New York, 1972.

\bibitem{ger} G.\ Gerboud, Substituabilit\'e d'un anneau de Dedekind, C.\ R.\ Acad.\ Sci.\ Paris S\'er.\ A 317 (1993) 29--32.

\bibitem{hal} P.\ Hall, {\it The Edmonton Notes on Nilpotent Groups}, Queen Mary College Mathematics Notes, Mathematics Department, Queen Mary College, London, 1969.

\bibitem{kru} W.\ Krull, {\it Idealtheorie}, Springer-Verlag, Berlin, 1935.

\bibitem{arm} A.\ Leriche, {\it Groupes, Corps et Extensions de P\'olya: une Question de Capitulation}, Ph.D.\ Thesis, Universit\'e de Picardie Jules Verne, 2010.

\bibitem{wil} C.\ Wilkerson,  $\lambda$-rings, binomial domains, and vector bundles over $CP(\infty)$, Comm.\ Algebra \ 10 (1982) 311--328.

\bibitem{zaf} M.\ Zafrullah, $t$-invertibility and Bazzoni-like statements, J.\ Pure Appl.\ Algebra 214 (5) (2010) 654--657.

\end{thebibliography}
\end{document}